\documentclass{amsart}
\usepackage{mathrsfs,amssymb,multirow,setspace,color,amsmath}
\usepackage[a4paper, total={7in, 9in}]{geometry}
\theoremstyle{plain}
\usepackage{graphicx}
\usepackage{mathtools}

\newtheorem{theorem}{Theorem}[section]
\newtheorem{lemma}[theorem]{Lemma}

\newtheorem{corollary}[theorem]{Corollary}

\theoremstyle{definition}

\theoremstyle{remark}
\newtheorem{remark}[theorem]{Remark}

\input xy
\xyoption{all}

\def\po{\mathcal{P}(G)}

\def\del{\Delta(G)}

\def\sg{\mathcal{S}(G)}
\def\ss{\mathcal{S}^*(G)}
\def\sss{\mathcal{S}^{**}(G)}

\begin{document}
\title[Line graph characterization of the order supergraph of a finite group]{Line graph characterization of the order supergraph of a finite group}



\author[Manisha, Parveen, Jitender Kumar]{Manisha, Parveen, $\text{JITENDER KUMAR}^{^*}$}
\address{$\text{}^1$Department of Mathematics, Birla Institute of Technology and Science Pilani, Pilani-333031, India}
\email{yadavmanisha2611@gmail.com,p.parveenkumar144@gmail.com,jitenderarora09@gmail.com}

\begin{abstract}
The power graph ${\mathcal{P}(G)}$ is the simple undirected graph with group elements as a vertex set and two elements are adjacent if one of them is a power of the other. The order supergraph $\sg$ of the power graph $\po$ is the simple undirected graph with vertex set $G$ in which two vertices $x$ and  $y$ are adjacent if $o(x)\vert o(y)$ or $o(y)\vert o(x)$. In this paper, we classify all the finite groups $G$ such that the order supergraph $\sg$ is the line graph of some graph. Moreover, we characterize finite groups whose order supergraphs are the complement of line graphs.
\end{abstract}

\keywords{Power graph, order supergraph of power graph, line graph, finite groups. \\ * Corresponding Author}

\maketitle

\section{Historical background}
The study of inter relationships between graphs and algebraic structures, viz: groups, rings, vector spaces etc; is a broad research area of algebraic graph theory. Various researcher studied the graphs associated with groups as they have enormous applications in the area of mathematics and automata theory (see \cite{kelarev2003graph, kelarev2004labelled,a.kelarev2009cayley}). In literature, there are various graphs associated with groups, e.g. Cayley graphs, commuting graphs, power graphs, prime graphs etc.. The concept of the directed power graph was introduced in \cite{a.kelarev2000groups}. The \emph{power graph} $\mathcal{P}(G)$ of a group $G$ is the simple undirected graph whose vertex set  is the corresponding set of $G$ and two vertices $a$ and $b$ are adjacent if one is a power of the other or equivalently: either $a \in  \langle b\rangle$ or $b \in \langle a\rangle$. In \cite{a.Cameron2010}, the author determined that the power graphs of two finite groups are isomorphic if they have the same number of elements of each order. Later, Cameron \emph{et. al} \cite{a.Cameron2011} proved that for finite abelian groups $G_1$ and $G_2$, $\mathcal{P}(G_1) \cong \mathcal{P}(G_2) $ if and only if $G_1 \cong G_2$. A graph $\Gamma$ is said to be $\Gamma'$-free if it contains no induced subgraph, which is isomorphic to $\Gamma'$. In  \cite{a.doostabadiforbidden}, Doostabadi \emph{et. al} classified all the finite groups with $K_{1,3}$-free, $K_{1,4}$-free or $C_4$-free power graphs. In \cite{a.MannaForbidden2021}, the authors studied certain forbidden subgraphs such as split, threshold, chordal and cograph of the power graphs of finite groups. For a detailed list of results and open problems on power graphs, we refer the reader to \cite{a.powergraphsurvey} and references therein.



The order supergraph $\sg$ of the power graph is the simple undirected graph whose vertex set is $G$ and two vertices $x, y \in G$ are adjacent if $o(x)\vert o(y)$ or $o(y)\vert o(x)$. The power graph $\po$ of a finite group $G$ is a spanning subgraph of $\sg$. Hamzeh and Ashrafi \cite{a.hamzeh2017automorphism} studied the automorphism groups of order supergraphs of certain finite groups. Further, in \cite{a.Hamzeh2018order}  they studied the relation between certain properties of the power graph and the order supergraph. Some essential properties of $\sg$ including Hamiltonianity, Eulerianness and $2$-connectedness have been studied in \cite{a.Hamzeh2019someremarks}. Ma \emph{et al.} \cite{a.Ma2022order} studied the independence number of the order supergraph. Also, they obtained a necessary and sufficient condition for the independence number of $\sg$ to be equal to the number of distinct prime divisors of the order of $G$. Asboi \emph{et al.} \cite{a.Asboei2022themainsupergraph} showed that a group is isomorphic to some simple groups, namely sporadic simple groups, PSL$(2,p)$, PGL$(2,p)$ if and only if their corresponding order supergraphs are isomorphic.

The \emph{line graph} $L(\Gamma)$ of the graph $\Gamma$ is the graph whose vertex set is all the edges of  $\Gamma$ and two vertices of $L(\Gamma)$ are adjacent if they are incident in $\Gamma$.
 Bera \cite{a.bera2022} characterized all the finite nilpotent groups whose power graphs and proper power graphs are line graphs. In \cite{a.kumar2023finite}, the authors have been classified all the finite groups whose enhanced power graphs are line graphs. Moreover, all finite nilpotent groups whose proper enhanced power graphs are line graphs of some graphs are determined in \cite{a.kumar2023finite}. In this paper, we aim to study the line graphs of order super power graphs associated to finite groups. The graphs $\ss$ and $\sss$ are the subgraphs of $\sg$ obtained by deleting the identity element of $G$ and all the dominating vertices of $\sg$, respectively. We characterize all the finite group $G$ such that $\del \in \{\sg, \ss, \sss \} $ is a line graph of some graph. Also, we classify all finite groups $G$ such that $\del \in \{\sg, \ss, \sss \} $ is the complement of a line graph.


\section{Preliminaries}
A \emph{graph} $\Gamma$ consists of an ordered pair of the vertex set $V(\Gamma)$ and the edge set $E(\Gamma)\subseteq V(\Gamma)\times V(\Gamma)$, in which two vertices $u$ and $v$ are adjacent if $\{u, v\}\in E(\Gamma)$. If $u$ is adjacent to $v$, then we denote it by $u \sim v$. Otherwise, $u \nsim v$.
If $\{u,v\} \in E(\Gamma)$, then the vertices $u$ and $v$ are called \emph{endpoints} of the edge $\{u,v\}$. Two edges $e_1$ and $e_2$ are called \emph{incident edges} if they have a common endpoint. An edge $e$ is called a \emph{loop} if both the endpoints of $e$ are the same. A graph is called a \emph{simple graph} if it does not contain any loop or multiple edges. We consider only simple graphs throughout the paper. 
 A graph $\Gamma$ is said to be an $\emph{empty graph}$ if the vertex set $V(\Gamma)$ is empty. A $\emph{subgraph}  \ \Gamma'$ of a graph $\Gamma$ is a graph  such that $V(\Gamma')\subseteq V(\Gamma)$ and $E(\Gamma')\subseteq E(\Gamma)$. 
A \emph{spanning subgraph} $\Gamma'$ of a graph $\Gamma$ is a subgraph of $\Gamma$ such that $V(\Gamma')=V(\Gamma)$. Let $X\subseteq V(\Gamma)$. Then the subgraph $\Gamma'$ \emph{induced by the set} $X$ is a graph such that $V(\Gamma')=X$ and $u,v\in X$ are adjacent if and only if they are adjacent in $\Gamma$.
If a vertex $u$ of a graph $\Gamma$ is adjacent to all other vertices of $\Gamma$ then $u$ is called a \emph{dominating vertex} of $\Gamma$. By $\mathrm{Dom}(\Gamma)$, we mean the set of all dominating vertices of $\Gamma$. A graph $\Gamma$ is said to be \emph{complete} if each pair of distinct vertices is adjacent. The complete graph of $n$ vertices is denoted by $K_n$. A graph $\Gamma$ is called a \emph{bipartite graph} if the vertex set $V(\Gamma)$ is partitioned into two subsets $V_1$ and $V_2$ such that every edge of $\Gamma$ has an endpoint in $V_1$ and one endpoint in $V_2$. A bipartite graph is said to be a \emph{complete bipartite} graph if each vertex of one partition is adjacent to every vertex of the other partition set. We denote by $K_{m,n}$ a complete bipartite graph with partition sizes $m$ and $n$. The complete bipartite graph $K_{1,n}$ is called the \emph{star graph}. The \emph{complement} of a graph $\Gamma$ is the graph $\overline{\Gamma}$ such that $V(\Gamma)= V(\overline{\Gamma})$ and two vertices $u$ and $v$ are adjacent in $\overline{\Gamma}$ if and only if $u$ is not adjacent to $v$ in $\Gamma$. A \emph{path} from $u$ to $v$ in a graph $\Gamma$ is a sequence of $r+1$ distinct vertices starting with $u$ and ending with $v$ such that consecutive vertices are adjacent. A graph $\Gamma$ is said to be \emph{connected} if there is a path between any pair of vertices of $\Gamma$. If a graph $\Gamma$ is equal to a path, then $\Gamma$ is called a \emph{path graph}. By $P_n$ we mean the path graph of $n$ vertices. Let $\Gamma_1,\ldots , \Gamma_m$ be $m$ graphs such that $V(\Gamma_i)\cap V(\Gamma_j)= \varnothing$, for distinct $i, j$. Then $\Gamma =\Gamma_1 \cup \cdots \cup \Gamma_m$ is a graph with vertex set  $V(\Gamma_1) \cup \cdots \cup V(\Gamma_m)$ and edge set $E(\Gamma_1) \cup \cdots \cup E(\Gamma_m)$. Let $\Gamma_1$ and $\Gamma _2$ be two graphs with disjoint vertex set, the \emph{join} $\Gamma_1 \vee \Gamma_2$ of $\Gamma_1$ and $\Gamma_2$ is the graph obtained from the union of $\Gamma_1$ and $\Gamma_2$ by adding new edges from each vertex of $\Gamma_1$ to every vertex of $\Gamma_2$.  Two graphs $\Gamma_1$ and $\Gamma _2$ are \emph{isomorphic} if there is a bijection $f$ from $V(\Gamma _1)$ to $V(\Gamma _2)$ such that if $u\sim v$ in $\Gamma _1$ if and only if $f(u)\sim f(v)$ in $\Gamma_2$.

The following two characterization of the line graph and the complement of the line graph are useful in the sequel.
\begin{lemma}{\rm \cite{a.beineke1970}}{\label{induced lemma}}
A graph $\Gamma$ is the line graph of some graph if and only if none of the nine graphs in $\mathrm{Figure \; \ref{fig line graph}}$  is an induced subgraph of $\Gamma$.
\end{lemma}
   \begin{figure}[ht]
    \centering
    \includegraphics[scale=.9]{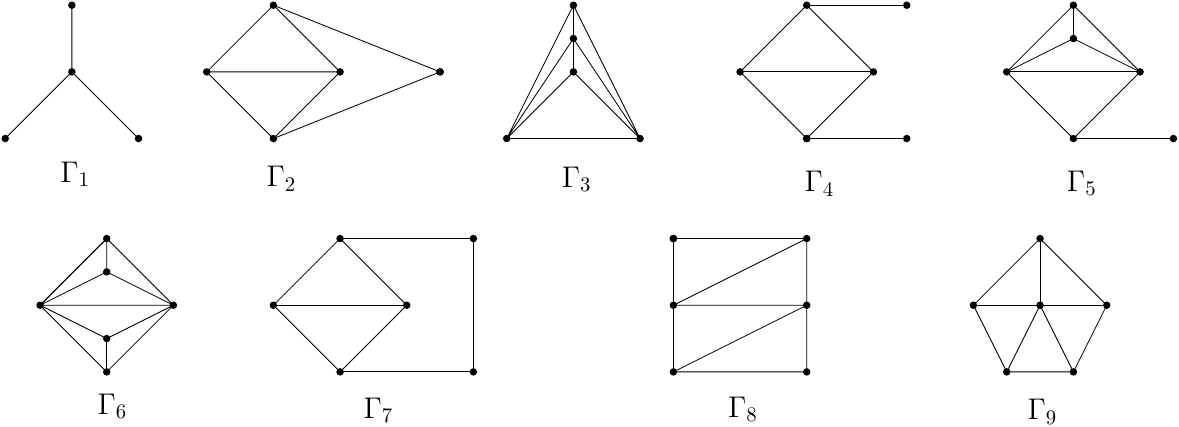}
    \caption{Forbidden induced subgraphs of line graphs.}
    \label{fig line graph}
\end{figure}

\begin{lemma}{\rm \cite[Theorem 3.1]{a.barati2021}}{\label{complement induced lemma}}
A graph $\Gamma$ is the complement of a line graph if and only if none of the nine graphs $\overline{\Gamma_i}$ in $\mathrm{Figure \; \ref{fig complement_line_graph}}$ is an induced subgraph of $\Gamma$.
   \begin{figure}[ht]
    \centering
    \includegraphics[scale=.9]{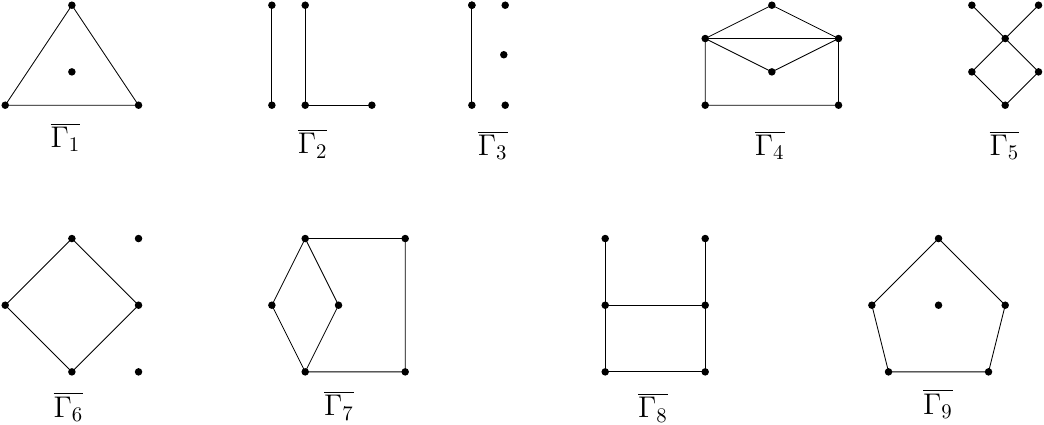}
    \caption{Forbidden induced subgraphs of the complement of line graphs.}
    \label{fig complement_line_graph}
\end{figure}
\end{lemma}
Let $G$ be a finite group. The order of an element $x\in G$ is denoted by $o(x)$. For a positive integer $n$, $\phi(n)$ denotes the Euler's totient function of $n$. An EPPO-group is a finite group in which the order of every element is a power of a prime. The Dihedral group $D_{2n}$ is a regular $n$-gon with $n$-rotational symmetries and $n$-reflectional symmetries mathematically represented as $D_{2n} = \langle x,y: x^{n} = e,y^{2}=e$ and $yxy^{-1}=x^{-1} \rangle$. The semi-dihedral group $SD_{8n} = \langle a,b : a^{4m}=b^{2}=e, ab= ba^{2m-1} \rangle$. The generalized quaternion group $Q_{4n} = \langle a, b: a^{n} = b^{2}, a^{2n} = 1, b^{1}ab = a^{-1} \rangle$. The exponent of $G$ is defined as the least common multiple of the orders of all the elements of $G$ and it is denoted by $\mathrm{exp}(G)$. Throughout this paper, $G$ is a finite group and $e$ is the identity element of $G$.
The following results are useful for later use.


\begin{theorem}{\rm \cite{b.dummit1991abstract}}{\label{nilpotent}}
 Let $G$ be a finite group. Then the following statements are equivalent:
 \begin{enumerate}
     \item[(i)] $G$ is a nilpotent group.
     \item[(ii)] Every Sylow subgroup of $G$ is normal.
    \item[(iii)] $G\cong P_1\times P_2 \times \cdots \times P_r$, where $P_i$'s are Sylow $p_i$-subgroups of $G$.
    \item[(iv)] For $x,y\in G, \  x$ and $y$ commute whenever $o(x)$ and $o(y)$ are relatively primes.
 \end{enumerate}
 \end{theorem}

 \begin{theorem}{\rm \cite[Theorem 2.3]{a.Hamzeh2018order}}\label{complete p group}
    Let $G$ be a finite group. Then the order supergraph $\sg$ is complete if and only if $G$ is a $p$-group.
 \end{theorem}
\section{Line graph characterization of $\sg$}
In this section, we classify all the groups $G$ such that $\sg$ is a line graph (see Theorem \ref{line SG graph}). Further, we determine all the group classes such that $\ss$ (see Theorem \ref{line S*G graph}) and $\sss$ (see Theorem \ref{line S**G graph}) are line graphs. Finally, we characterize all the groups $G$ such that $\sg$, $\ss$ and $\sss$ are the complement of the line graph of some graph (see Theorem \ref{line complement SG graph}).
The set $\{1,2,\ldots , k\}$ is denoted by $[k]$. 

\begin{theorem}{\label{line SG graph}}
The order supergraph $\sg$ is a line graph of some graph $\Gamma$ if and only if $G$ is an EPPO-group and the order of $G$ is divisible by at most two distinct primes.
\end{theorem}

\begin{proof}
    First, assume that $\sg$ is a line graph. If $|G| = p_{1}^{\alpha_{1}}p_{2}^{\alpha_{2}}\cdots p_{k}^{\alpha_{k}}$, where 
$k\geq{3}$, then consider $x,y,z\in G$ such that $o(x) = p_{1}$, $o(y) = p _2$ and $o(z) = p_3$. Note that the subgraph of $\sg$ induced by the set $\{ e, x, y, z \}$ is isomorphic to $\Gamma_{1}$. This implies that $k \leq 2$ and so $o(G)$ is divisible by at most two primes. Now suppose there exists an element of order $p_1p_2$. Consider $x, y, z, w\in G$ such that $o(x) = p_{1}, \ o(y)=o(z)= p_{1}p_{2}$ and $ o(w)= p_{2}$. Then the subgraph induced by the set $\{e,x, y, z, w\}$ is isomorphic to  $\Gamma_3$ (see Figure \ref{fig line graph}); a contradiction. Thus,  $G$ is an EPPO-group.

Conversely, suppose $G$ is an EPPO-group and $o(G)$ is divisible by at most two primes. If $G$ is a $p$-group, then $\sg$ is complete and so $\sg$ is line graph.
Now suppose $|G| = p_{1}^{\alpha_{1}}p_{2}^{\alpha_{2}}$. Assume that $\Gamma$ is an induced subgraph of $\sg$ such that $\Gamma \cong \Gamma _{i}$ for some $i$, where $2 \leq i \leq 9$. Note that $\Gamma$ has an induced subgraph isomorphic to $\Gamma^{'}$ as shown in Figure \ref{fig gamma'}. Since $x \nsim w$, we obtain $x \neq e$. Therefore, $o(x) = p_{1}^{a}$ or $ o(x) = p_{2}^{b}$.
Without loss of generality, let $o(x) = p_{1}^{a}$. Since $x\sim y$ it follows that $o(y) = 1$ or $p_1^{a_1}$. If $y = e$, then $z \neq e$. Since $x \sim z$, we get $o(z) = p_1^{a_2}$. Consequently $z \sim w$ gives $o(w) = p_1^{s}$. Then either $o(x) \vert o(w)$ or $o(w) \vert o(x)$. Consequently, $x \sim w $; a contradiction. If $o(y) = p_1^{a_1}$, then $y \sim w $ implies  that $o(w) = p_{1}^{a_2}$ or $1$. Since $x$ is not adjacent to $w$, therefore $o(w) \neq 1$. Now if $o(w) = p_1^{a_3}$, then $w \sim x$; a contradiction. Thus, $\Gamma$ can not be an induced subgraph of $\sg$.\\
Now if $\Gamma \cong K_{1,3}$ as shown in Figure \ref{fig gamma'}. If $o(d)=1$, then $o(a) \in \{ p_{1}^{\alpha},p_{2}^{\beta}\}$. Without loss of generality, assume that $o(a)= p_1^{\alpha}$. Since $a\nsim b$, it follows that $o(b)=p_2^{\gamma}$. Observe that $o(c) \in \{ p_{1}^{r},p_{2}^{s}\}$. Consequently, either $a\sim c$ or $b\sim c$ which is not possible. Thus, $o(d)\neq 1$ and so $o(d) \in 
 \{ p_{1}^{t},p_{2}^{t^{'}} \}$. Without loss of generality, assume that $o(d)= p_1^{t}$. Since  $a \sim d$ and $d \sim b$ it imply that $o(a), \ o(b)$ and $ \ o(d)$ are divisors of $p_1^{\alpha _1}$. Consequently, $a \sim b$; a contradiction. Thus, $\Gamma_{1}$ cannot be an induced subgraph of $\sg$. Hence, $\sg$ is a line graph.
   \begin{figure}[ht]
    \centering
    \includegraphics[scale=.7]{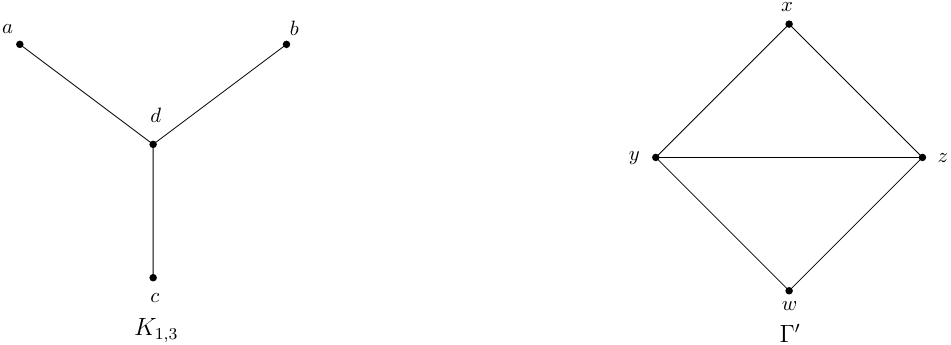}
    \caption{Forbidden induced subgraphs of line graphs.}
    \label{fig gamma'}
\end{figure}
\end{proof}

Let $G$ be a nilpotent group. Then note that $G$ is an EPPO-group if and only if $G$ is a $p$-group. Thus, we have the following corollary of Theorem \ref{line SG graph}.

\begin{corollary}{\label{cor line nilpotent SG graph}}
    Let $G$ be a finite nilpotent group. Then $\sg$ is a line graph of some graph $\Gamma$ if and only if $G$ is a $p$-group.
\end{corollary}

\begin{corollary}{\label{cor line dihedral SG graph}}
    Let $D_{2n}$ be the dihedral group of order $2n$. Then $\mathcal{S}(D_{2n})$ is a line graph of some graph $\Gamma$ if and only if $n = p^{\alpha}$ for some prime $p$ and $\alpha \in \mathbb{N}$.
\end{corollary}

\begin{proof}
    Let $\mathcal{S}(D_{2n})$ be a line graph and let $n$ is not a power of a prime. Note that $D_{2n}$ has a cyclic subgroup of order $n$ and so $G$ contains an element of order $pq$, where $p, q$ are distinct prime divisors of $n$. Consequently, $G$ is not an EPPO-group; a contradiction to Theorem \ref{line SG graph}.

    Conversely, if $n = p^{\alpha}$ then $D_{2n}$ is an EPPO-group and therefore $\mathcal{S}(D_{2n})$ is a line graph of some graph $\Gamma$.
\end{proof}

  \begin{theorem}{\label{line S*G graph}}
 The proper order supergraph $\ss$ is a line graph of some graph $\Gamma$ if and only if either $G\cong \mathbb{Z}_6$ or $G$ is an EPPO-group.
\end{theorem}

\begin{proof}
First, assume that $\ss$ is a line graph. Let $G$ contains an element of order $d \ (>6)$ which is not a power of a prime. Then $G$ has at least three elements of order $d$. Consider $x_1,x_2,x_3,y_1,y_2 \in G$ such that $o(x_1) = o(x_2) = o(x_3) = d, \ o(y_1) = p$ and $o(y_2) = q$, where $p$ and $q$ are distinct prime divisors of $d$. Then the subgraph of $\ss$ induced by the set $\{x_1,x_2,x_3,y_1,y_2\}$ is isomorphic to $\Gamma_{3}$; a contradiction. Thus, $G$ cannot have an element of order $d$.

 Now suppose that $G$ has an element of order 6. Further we have the following two cases. \\
\noindent\textbf{Case-1:} 
\textit{ $G$ has more than one cyclic subgroup of order $6$}. In this case $G$ has at least four elements of order $6$. Consider ${x_1, x_2, x_3, y_1,y_2} \in G$ such that $o(x_1) = o(x_2) = o(x_3) = 6, \  o(y_1)=2$ and $o(y_2)= 3$. Then the subgraph of $ \ss$ induced by the set $\{x_1, x_2, x_3, y_1,y_2\}$ is isomorphic to $\Gamma_3$; a contradiction. Therefore, this case is not possible.\\
\noindent\textbf{Case-2:} 
\textit{$G$ has exactly one cyclic subgroup of order $6$}. In this case, we prove that $G$ is isomorphic to $\mathbb{Z}_6$. Let $H = \langle x \rangle$ be the unique cyclic subgroup of order $6$ in $G$. Then $g^{-1}Hg = H$ for $g \in G$ and so $H$ is a normal subgroup of $G$. Now we claim that $C_{G}(x) = \langle x \rangle$. Clearly $\langle x \rangle$ $\subseteq$ $C_{G}(x)$. Let $y$ $\in$ $C_{G}(x)\setminus\langle x \rangle$. Then $o(y)$ is the power of a prime. Consider $o(y) = p^{\alpha}$, for some prime $p$. If $\mathrm{gcd}(6,p) = 1$, then $\langle xy \rangle$ is a cyclic subgroup of order $6p^{\alpha}$ which is not possible. If $p = 2$, then note that $yx^2 = x^2y $ and $o(x^2) = 3$. Consequently, $G$ has a cyclic subgroup of order $3.2^{\alpha}$ containing $y$, which is not possible. By using a similar argument, we obtain a contradiction for $p = 3$. This proves our claim.
    Thus, $\langle x \rangle$ is a normal subgroup of $G$ and $C_{G}(x) = \langle x \rangle $. For a normal subgroup $H$, it is known that $\frac{G}{C_{G}(H)}$ is a subgroup of $\mathrm{Aut}(H)$. Thus, $o(G) \in \{6,12\}$. Therefore, $G$ is isomorphic to one of the three groups: $ \mathbb{Z}_6$, ${D}_{12}$, ${Q}_{12}$. 

      If $G \cong {D}_{12}$, then $G$ has seven elements of order $2$, two elements of order $3$ and two elements of order $6$. Consider $x_1,x_2,y_1,y_2,z_1,z_2$ such that $o(x_1) = o(x_2) = 2$, $o(y_1) =o(y_2) = 3$ and $o(z_1)=o(z_2) = 6$. The subgraph of $\ss$ induced by the set $\{x_1,x_2,y_1,y_2,z_1,z_2\}$ is isomorphic to $\Gamma_6$; a contradiction.
     
     If $G \cong {Q}_{12}$, then $G$ has two elements of order $6$, two elements of order $3$, one element of order $2$ and six elements of order $4$. Let $x_1,x_2,y_1,y_2,z_1,z_2 \in G$  be such that $o(x_1) = o(x_2) = 6$, $o(y_1) =o(y_2) = 3$, $o(z_1) = 2$ and $o(z_2) = 4$. The subgraph of $\ss$ induced by the set $\{x_1,x_2,y_1,y_2,z_1,z_2\}$ is isomorphic to $\Gamma_5$, a contradiction. It follows that $G \cong \mathbb{Z}_6$.

    Conversely, if $G \cong \mathbb{Z}_6$ then note that $\ss$ is a line graph of the graph $\Gamma$ (see Figure \ref{fig Sz6}(a)). 
     \begin{figure}[ht]
    \centering
    \includegraphics[scale=.9]{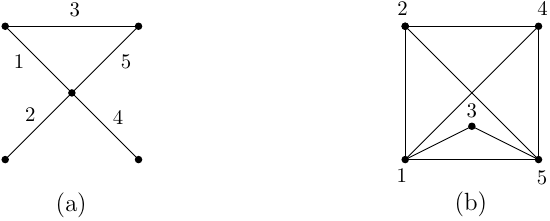}
    \caption{(a) The graph $\Gamma$ (b) $L(\Gamma)= \mathcal{S}^*(\mathbb{Z}_6)$. }
    \label{fig Sz6}
\end{figure}
    Now let $G$ be an EPPO-group, where $|G| = p_1^{{\alpha}_{1}}\cdots p_{k}^{{\alpha}_{k}}$. Then $\ss$ is the disjoint union of complete graphs $K_{n_1},K_{n_2},\ldots$ , $K_{n_k}$, respectively, where $n_i ( 1\leq i \leq k)$ is the number of elements whose order is divisible by $p_{i}$. Note that $\ss = L( K_{1,n_1} \cup K_{1,n_2} \cup \cdots \cup K_{1,n_k})$. This completes our proof.
        
\end{proof}
The order supergraph $\sg$ is dominatable if $\sg$ has a dominating vertex other than identity. 
In order to prove the  Theorem \ref{line S**G graph}, we require the following lemma.

\begin{lemma}{\label{dominatable}}
    The graph $\sg$ is dominatable if and only if there exists an element $x$ in $G$ such that $o(x) = \mathrm{exp}(G)$. Moreover, $x$ is a dominating vertex in $\sg$.
\end{lemma}

\begin{proof}
If $G$ is a $p$-group, then $\sg$ is complete. Thus, the result holds trivially. We may now suppose that $G$ is not a $p$-group. Let $\sg$ be a dominatable graph. Then there exists a non-identity element $x$ of $G$ such that $x$ is adjacent to every other element of $G$. We show that $o(x)= \mathrm{exp}(G)$. Assume that $\mathrm{exp}(G)= p_1^{\alpha_1}p_2^{\alpha_2}\cdots{p_k^{\alpha_k}}$. On contrary, let $o(x) \neq \mathrm{exp}(G)$. Then there exists $i \in [k] $ such that $p_i^{\alpha_{i}} \vert o(x)$. Consider $g \in G$ such that $o(g)=p_i^{\alpha_{i}}$. It follows that $x \sim g $ if and only if $o(x) = p_i^{\beta_{i}}$ for some $\beta_i < \alpha_i$. But $o(x) = p_i^{\beta_{i}}$ implies that $x$ is not adjacent to the element $z$ such that $o(z)= p_j$ for some $i \neq j \in [k]$; a contradiction. Conversely, let $G$ contains an element $x$ such that $o(x) = \mathrm{exp}(G)$. Then for any $y \in G$, we have $o(y)\vert o(x)$. Thus, $y \sim x$ in $\sg$ and so $\sg$ is dominatable.
\end{proof}

\begin{remark}
    If $G$ is a finite group which is not a $p$-group, then $x \ (\neq e) \in G$ is a dominating vertex of $\sg$ if and only if $o(x) = \mathrm{exp}(G)$. Thus, $V(\sss) = G\setminus (S\cup \{e \})$, where $S =\{ y \in G: o(y) = \mathrm{exp}(G)\}$.
\end{remark}

\begin{theorem}{\label{line S**G graph}}
    Let $G$ be a finite group such that $\sg$ is dominatable. Then $\sss$ is a line graph of some graph $\Gamma$ if and only if one of the following holds:
    \begin{itemize}
        \item[(i)] $G$ is a $p$-group.
        \item[(ii)] The order of $G$ is divisible by two primes and the order of each element of $G$ is square-free.
    \end{itemize}
\end{theorem}

\begin{proof}
Let $\sss$ be a line graph and $|G| = p_1^{\alpha_1}p_2^{\alpha_2}\cdots{p_k^{\alpha_k}}$, where $p_1< p_2< \cdots < p_k$ are primes. If $k = 1$ then $G$ is a $p$-group. Now we discuss the following cases.


\noindent\textbf{Case-1:} $k =2$. Let $\mathrm{exp}(G) = p_1^{\beta_1}p_2^{\beta_2}$. If $\beta_1, \beta_2 = 1$, then we obtain condition (ii). We may now suppose that $\beta_i > 1$ for some $i \in \{1,2\}$. Now we discuss the following subcases:

\textbf{Subcase-1.1:} $\beta_1 > 1$. Consider $x_1,x_2,y_1,y_2,z_1,z_2 \in G$ such that $o(x_1) = p_1, 
  \ o(x_2) = p_1^{2}, \  o(y_1)=o(y_2) = p_2$ and $o(z_1)=o(z_2) = p_1p_2$. The subgraph of $\sss$ induced by the set $\{ x_1,x_2,y_1,y_2,z_1,z_2\}$ is isomorphic to $\Gamma_5$; a contradiction.

\textbf{Subcase-1.2:} $\beta_2 > 1$. Consider $x_1,x_2,y_1,y_2,z_1,z_2 \in G$ such that $o(x_1) =  o(x_2) = p_2, \ o(y_1)=o(y_2) = p_1p_2$ and $o(z_1)=o(z_2) = p_2^{2}$. The subgraph of $\sss$ induced by the set $\{ x_1,x_2,y_1,y_2,z_1,z_2\}$ is isomorphic to $\Gamma_6$; a contradiction.

\noindent\textbf{Case-3:} $k \geq 3$. Let $\mathrm{exp}(G) = p_1^{\beta_1}p_2^{\beta_2}\cdots{p_k}^{\beta_k}$, for some $1 \leq \beta_i \leq \alpha_{i}$ for $i \in [k]$. Let $x_1,x_2,y_1,y_2,z_1,z_2 \in G$ be such that $o(x_1) = o(x_2) = p_2, \  o(y_1)= o(y_2) = p_3$ and $o(z_1) =o(z_2) = p_2p_3$. The subgraph of $\sss$ induced by the set $\{ x_1,x_2,y_1,y_2,z_1,z_2\}$
 is isomorphic to $\Gamma_6$, a contradiction.

 Conversely, if $G$ is a $p$-group then $\sss$ is the empty graph and so is a line graph. If $G$ satisfies condition (ii), then $\pi_G = \{ 1, p, q, pq\}$. Note that $\sss \cong K_{|S_1|} \cup K_{|S_2|}$, where $S_1$ is the set of elements of order $p$ and $S_2$ is the set of elements of order $q$. Observe that $K_{|S_1|} \cup K_{|S_2|} = L(K_{1,{|S_1|}} \cup K_{1,{|S_2|}})$. This completes our proof.
\end{proof}

\begin{corollary}{\label{cor line Q4n S**G graph}}
    Let $G$ be the generalized quaternion group $Q_{4n}$. Then $\mathcal{S}^{**}(Q_{4n})$ is a line graph of some graph $\Gamma$ if and only if $n=2^k$ for some $k\in \mathbb{N}$.
\end{corollary}

\begin{proof}
     Since $Q_{4n}$ contains an element of order $4$, therefore, by Theorem \ref{line S**G graph}, $\mathcal{S^{**}}(Q_{4n})$ is a line graph if and only if $n = 2^k$.
\end{proof}

\begin{corollary}{\label{cor line SD8n S**G graph}}
    Let $G$ be the semidihedral group $SD_{8n}$. Then $\mathcal{S}^{**}(SD_{8n})$ is a line graph of some graph $\Gamma$ if and only if $n=2^k$ for some $k\geq 2$.
\end{corollary}

    \begin{proof}
          Since $SD_{8n}$ contains an element of order $8$, therefore, by Theorem \ref{line S**G graph}, $\mathcal{S^{**}}(SD_{8n})$ is a line graph if and only if $n = 2^k$.
         
    \end{proof}

 \begin{theorem}{\label{line complement SG graph}}
    Let $G$ be a finite group and let $\Delta(G)\in \{\sg, \  \ss, \  \sss \}$. Then $\Delta(G)$ is the complement of the line graph if and only if either $G\cong \mathbb{Z}_6$ or $G$ is a $p$-group.
\end{theorem}

\begin{proof}
Let $\Delta(G) $ be the complement of the line graph of some graph $\Gamma$. If $G$ is a $p$-group, then the result holds. Suppose $G$ is not a $ p$-group. Then we have the following cases.

\noindent\textbf{Case-1:} \emph{$p\vert o(G)$, where $p > 3$ is a prime}. In this case, we have at least four elements of order $p$ in $G$. Let $x, y, z \in G$ such that $ o(x)=o(y)=o(z) = p$. Since $G$ is not a $p$-group. Consequently, it is divisible  by an another prime $q$. Consider $w \in G$ such that $o(w) = q$. The subgraph of $\Delta(G)$ induced by the set $\{x,y,z,w\}$ is isomorphic to $\overline{\Gamma_1}$, which is a contradiction (see Figure \ref{fig complement_line_graph}).

\noindent\textbf{Case-2:} $o(G) = 2^{\alpha}3^{\beta}$ \emph{for some} $\alpha,\beta \in \mathbb{N}$. Without loss of generality, assume that $\beta \geq 2$. Let $H$ be the Sylow subgroup of $G$ such that $o(H) =  3^{\beta}$. Consider three non-identity elements $x,y,z$ of $H$. Let $w \in G$  such that $o(w) = 2$. Then the subgraph of $\Delta(G)$ induced by the set $\{x,y,z,w\}$ is isomorphic to $\overline{\Gamma_{1}}$, a contradiction. 
Thus, $o(G) = 6 $. If $G \cong {S}_3$, then it has an induced subgraph which is isomorphic to $\overline{\Gamma_1}$.

  \begin{figure}[ht]
    \centering
    \includegraphics[scale=.9]{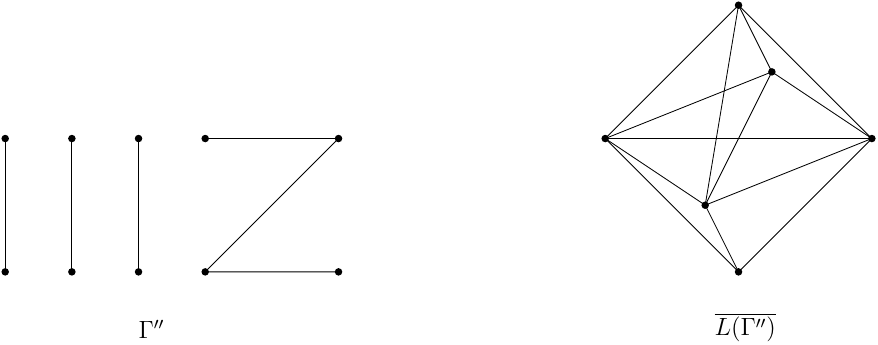}
    \caption{The graph $\Gamma''$ and the complement of the line graph of $\Gamma''$.}
    \label{fig Z6complement}
\end{figure} 
Conversely, if $G$ is a $ p$-group, then $\sg$  and $\ss$ are complete  graph, whereas $\sss$ is the empty graph (cf. Theorem \ref{complete p group}). Thus, $\sg = \overline{L(nK_{2})}$ and $\ss = \overline{L((n-1)K_{2})}$. If $G \cong \mathbb{Z}_6$, then $\mathcal{S}(\mathbb{Z}_6)$ is isomorphic to the complement of the line graph of a graph $\Gamma''$ (see Figure \ref{fig Z6complement}). Since $\ss$ and $\sss$ are induced subgraphs of $\sg$. Consequently, being induced subgraphs $\ss$ and $\sss$ are also the complement of the line graph of the induced subgraphs of $\Gamma^{''}$ (cf. Figure \ref{fig Z6complement}). This completes the proof.
\end{proof}



\section*{Declarations}

\textbf{Funding}: The first author and the third author wish to acknowledge the support of Core Research Grant (CRG/2022/001142) funded by  SERB, Government of India. The second author gratefully acknowledges for providing financial support to CSIR  (09/719(0110)/2019-EMR-I), Government of India.  

\vspace{.3cm}
\textbf{Conflicts of interest/Competing interests}: There is no conflict of interest regarding the publishing of this paper. 

\vspace{.3cm}
\textbf{Availability of data and material (data transparency)}: Not applicable.

\vspace{.3cm}
\textbf{Code availability (software application or custom code)}: Not applicable.



\begin{thebibliography}{10}

\bibitem{a.Asboei2022themainsupergraph}
A.~K. Asboei and S.~S. Salehi~Amiri.
\newblock The main supergraph of finite groups.
\newblock {\em New York J. Math.}, 28:1057--1063, 2022.

\bibitem{a.barati2021}
Z.~Barati.
\newblock Line zero divisor graphs.
\newblock {\em J. Algebra Appl.}, 20(9):2150154, 13, 2021.

\bibitem{a.beineke1970}
L.~W. Beineke.
\newblock Characterizations of derived graphs.
\newblock {\em J. Combin. Theory}, 9:129--135, 1970.

\bibitem{a.bera2022}
S.~Bera.
\newblock Line graph characterization of power graphs of finite nilpotent
  groups.
\newblock {\em Comm. Algebra}, 50(11):4652--4668, 2022.

\bibitem{a.Cameron2010}
P.~J. Cameron.
\newblock The power graph of a finite group, {II}.
\newblock {\em J. Group Theory}, 13(6):779--783, 2010.

\bibitem{a.Cameron2011}
P.~J. Cameron and S.~Ghosh.
\newblock The power graph of a finite group.
\newblock {\em Discrete Math.}, 311(13):1220--1222, 2011.

\bibitem{a.doostabadiforbidden}
A.~Doostabadi, A.~Erfanian, and M.~Farrokhi D.~G.
\newblock On power graphs of finite groups with forbidden induced subgraphs.
\newblock {\em Indag. Math. (N.S.)}, 25(3):525--533, 2014.

\bibitem{b.dummit1991abstract}
D.~S. Dummit and R.~M. Foote.
\newblock {\em Abstract algebra}.
\newblock Prentice Hall, Inc., Englewood Cliffs, NJ, 1991.

\bibitem{a.hamzeh2017automorphism}
A.~Hamzeh and A.~R. Ashrafi.
\newblock Automorphism groups of supergraphs of the power graph of a finite
  group.
\newblock {\em European J. Combin.}, 60:82--88, 2017.

\bibitem{a.Hamzeh2018order}
A.~Hamzeh and A.~R. Ashrafi.
\newblock The order supergraph of the power graph of a finite group.
\newblock {\em Turkish J. Math.}, 42(4):1978--1989, 2018.

\bibitem{a.Hamzeh2019someremarks}
A.~Hamzeh and A.~R. Ashrafi.
\newblock Some remarks on the order supergraph of the power graph of a finite
  group.
\newblock {\em Int. Electron. J. Algebra}, 26:1--12, 2019.

\bibitem{a.kelarev2000groups}
A.~Kelarev and S.~Quinn.
\newblock A combinatorial property and power graphs of groups.
\newblock {\em Contrib. General Algebra}, 12(58):3--6, 2000.

\bibitem{kelarev2003graph}
A.~V. Kelarev.
\newblock {\em Graph algebras and automata}, volume 257.
\newblock Marcel Dekker, Inc., New York, 2003.

\bibitem{kelarev2004labelled}
A.~V. Kelarev.
\newblock Labelled {C}ayley graphs and minimal automata.
\newblock {\em Australas. J. Combin.}, 30:95--101, 2004.

\bibitem{a.kelarev2009cayley}
A.~V. Kelarev, J.~Ryan, and J.~Yearwood.
\newblock Cayley graphs as classifiers for data mining: the influence of
  asymmetries.
\newblock {\em Discrete Math.}, 309(17):5360--5369, 2009.

\bibitem{a.powergraphsurvey}
A.~Kumar, L.~Selvaganesh, P.~J. Cameron, and T.~Tamizh~Chelvam.
\newblock Recent developments on the power graph of finite groups---a survey.
\newblock {\em AKCE Int. J. Graphs Comb.}, 18(2):65--94, 2021.

\bibitem{a.Ma2022order}
X.~Ma and H.~Su.
\newblock On the order supergraph of the power graph of a finite group.
\newblock {\em Ric. Mat.}, 71(2):381--390, 2022.

\bibitem{a.MannaForbidden2021}
P.~Manna, P.~J. Cameron, and R.~Mehatari.
\newblock Forbidden subgraphs of power graphs.
\newblock {\em Electron. J. Combin.}, 28(3):3.4, 14, 2021.

\bibitem{a.kumar2023finite}
Parveen and J.~Kumar.
\newblock On finite groups whose power graphs are line graphs.
\newblock {\em \rm{arXiv}:2307.01661}, 2023.

\end{thebibliography}

\vspace{1cm}
\noindent
{\bf Manisha\textsuperscript{\normalfont 1}, {\bf Parveen\textsuperscript{\normalfont 1}, \bf Jitender Kumar\textsuperscript{\normalfont 1}}
\bigskip

\noindent{\bf Addresses}:


\end{document}